\documentclass[smallextended]{my-svjour3}

\smartqed

\usepackage{graphicx}
\usepackage{mathptmx}
\usepackage{a4wide}
\usepackage{algorithm}
\usepackage{amsmath}
\usepackage{amssymb}
\usepackage{epstopdf}
\usepackage[dvips]{epsfig}
\usepackage{subfig}
\usepackage[misc,geometry]{ifsym} 
    
\begin{document}

\title{Analysis and optimal control of an intracellular 
delayed HIV model\\ with CTL immune response}

\titlerunning{Analysis and optimal control of an intracellular delayed HIV model}

% --------------------------------------

\author{Karam Allali \and Sanaa Harroudi \and Delfim F. M. Torres}

\authorrunning{K. Allali \and S. Harroudi \and D. F. M. Torres}

\institute{K. Allali \and S. Harroudi \at
Laboratory of Mathematics and Applications,
Faculty of Sciences and Technologies,\\
University Hassan II of Casablanca,
P.O. Box 146, Mohammedia, Morocco\\
\email{allali@hotmail.com}
\and
S. Harroudi\at
\email{sanaa.harroudi@gmail.com}
\and
D. F. M. Torres (\Letter) \at
Center for Research and Development in Mathematics and Applications (CIDMA),\\
Department of Mathematics, University of Aveiro, 3810-193 Aveiro, Portugal\\
\email{delfim@ua.pt}
}

% --------------------------------------

\date{Received: 30 May 2017 / Revised: 22 January 2018 / Accepted: 30 January 2018}

\maketitle

% --------------------------------------

\begin{abstract}
A delayed model describing the dynamics of HIV 
(Human Immunodeficiency Virus) with CTL
(Cytotoxic T Lymphocytes) immune response is investigated. 
The model includes four nonlinear differential equations 
describing the evolution of uninfected, infected, free HIV viruses, 
and CTL immune response cells. It includes also
intracellular delay and two treatments (two controls). 
While the aim of first treatment consists to block 
the viral proliferation, the role of the second 
is to prevent new infections. Firstly, we prove the
well-posedness of the problem by establishing some positivity and
boundedness results. Next, we give some conditions that insure the
local asymptotic stability of the endemic and disease-free equilibria.
Finally, an optimal control problem, associated with the intracellular
delayed HIV model with CTL immune response, is posed and investigated.
The problem is shown to have an unique solution, which is characterized
via Pontryagin's minimum principle for problems with delays.
Numerical simulations are performed, confirming stability of the
disease-free and endemic equilibria and illustrating the effectiveness of
the two incorporated treatments via optimal control.


\end{abstract}

% --------------------------------------

\section{Introduction}

Human immunodeficiency virus (HIV) is recognized as a viral
pathogen causing the well known acquired immunodeficiency syndrome
(AIDS), which is considered the end-stage of the infection.
After this stage, the immune system fails to play its principal
role, which is to protect the whole body against harmful intruders. 
This failure is due to destruction of the vast majority of CD4+ T 
cells by the HIV virus, reducing them to an account below $200$ 
cells per $\mu l$ \cite{2,1}. 

During last decades, many mathematical models have been developed in order
to better understand the dynamics of the HIV disease \cite{5a,3,4,5b}. 
Mathematical models of HIV and tuberculosis
coinfection have been investigated 
in \cite{MyID:300,MyID:318}. An interesting
case study, with real data from Cape Verde islands,
has been carried out in \cite{MyID:359},
showing that the goal of the United Nations 
to end the AIDS epidemic by 2030 is a nontrivial task. 
For the importance of optimization techniques
and optimal control in the study of HIV,
we refer the reader to \cite{MyID:366,MyID:367} 
and references therein. Here we observe that, often, 
models introduce the effect of cellular immune response, 
also called the cytotoxic T-lymphocyte (CTL) response, 
which attacks and kills the infected cells \cite{rst}. 
It has been shown that this cellular 
immune response can control the load of HIV viruses
\cite{perel2,perel1}. In \cite{crs}, it is assumed that
CTL proliferation depends, besides infected cells, as usual, 
also on healthy cells. Moreover, an optimal control problem associated 
with the suggested model is studied \cite{crs}. Recently, the same problem 
was tackled by introducing time delays \cite{rst}. Here, we
continue the investigation of such kind of problems by introducing 
the HIV virus dynamics to the system of equations. This is important 
because uninfected cells must be in contact with the HIV virus 
before they become infected. The proposed basic model, illustrating 
this type of scenario, is as follows:
\begin{equation}
\label{sy}
\begin{cases}
\displaystyle \frac{dx(t)}{dt} = \lambda -d x(t)- \beta x(t) v(t),\\[0.3cm]
\displaystyle \frac{dy(t)}{dt} = \beta x(t) v(t) - a y(t) -p y(t) z(t),\\[0.3cm]
\displaystyle \frac{dv(t)}{dt} = a N y(t) - \mu v(t), \\[0.3cm]
\displaystyle \frac{dz(t)}{dt} = c x(t) y(t) z(t) - h z(t),
\end{cases}
\end{equation}
subject to given initial conditions $x(0) = x_{0}$, $y(0) = y_{0}$,
$v(0) = v_{0}$, and $z(0)= z_{0}$. In this model, $x(t)$, $y(t)$, $v(t)$ 
and $z(t)$ denote, respectively, the concentrations  at time $t$ 
of uninfected cells, infected cells, HIV virus, and CTL cells. 
The healthy CD$4^{+}$ cells grow at a rate $\lambda$,
decay at a rate $d x(t)$ and become infected by the virus 
at a rate $\beta x(t) v(t)$. Infected cells $(y)$ 
die at a rate $a$ and are killed by the CTL
response at a rate $p$. Free virus $(v)$ is produced by the
infected cells at a rate $a N$ and decay at a rate $\mu$, where $N$
is the number of free virus produced by each actively infected
cell during its life time. Finally, CTLs $(z)$ expand in response
to viral antigen derived from infected cells at a rate $c$ and
decay in the absence of antigenic stimulation at a rate $h$.

The paper is organized as follows. Section~\ref{sec:2} is devoted 
to the proof of existence, positivity and boundedness of solutions. 
Then, in Section~\ref{sec:3}, we do an optimization analysis 
of the viral infection model. In Section~\ref{sec:4}, we construct 
an appropriate numerical algorithm and give some simulations. 
Finally, conclusions are given in Section~\ref{sec:5}.

% --------------------------------------

\section{Analysis of the model with delay}
\label{sec:2}

In order to be realistic, let us introduce an 
intracellular time delay to the system 
of equations \eqref{sy}. Then, the model 
takes the following form:
\begin{equation}
\label{sys}
\begin{cases}
\displaystyle \dot{x}(t) = \lambda -d x(t)- \beta x(t) v(t), \\
\displaystyle \dot{y}(t) = \beta x(t-\tau) v(t-\tau)- a y(t)- p y(t) z(t),\\
\displaystyle \dot{v}(t) = a N y(t)- \mu v(t), \\
\displaystyle \dot{z}(t) = c x(t) y(t) z(t) - h z(t).
\end{cases}
\end{equation}
Here, the delay $\tau$ represents the time needed
for infected cells to produce virions after viral entry. 
Model \eqref{sys} is a system of delayed ordinary
differential equations. For such kind of problems, initial
functions need to be addressed and an appropriate functional
framework needs to be specified. Let us first consider 
$X = C([-\tau,0];\mathbb{R}^4)$ to be the Banach space of continuous
mappings from $[-\tau,0]$ to $\mathbb{R}^4$ equipped with the sup-norm 
$\displaystyle \| \varphi \| = \sup_{-\tau \le t \le 0} |\varphi(t)|$. 
We assume that the initial functions verify 
\begin{equation}
\label{2}
\left(x(\theta), y(\theta), v(\theta), z(\theta)\right)\in X.
\end{equation}
Also, from biological reasons, these initial functions $x(\theta)$,
$y(\theta)$, $v(\theta)$ and $z(\theta)$ have to be nonnegative:
\begin{equation}
\label{3}
x(\theta) \ge 0, \quad y(\theta) \ge 0, \quad v(\theta) \ge 0, 
\quad z(\theta) \ge 0, \quad \text{ for } \theta \in [-\tau, 0].
\end{equation}

% ------------------

\subsection{Positivity and boundedness of solutions}
\label{sec:2:1}

For the solutions of \eqref{sys} with initial functions satisfying
conditions \eqref{2} and \eqref{3}, the following theorem holds.

\begin{theorem}
\label{thm1}
For any initial conditions $\left(x(t), y(t), v(t), z(t)\right)$ 
satisfying \eqref{2} and \eqref{3}, the system \eqref{sys} 
has a unique solution. In addition, the solution is nonnegative 
and bounded for all $t \ge 0$.
\end{theorem}

\begin{proof}
By the standard functional framework of ordinary
differential equations (see, for instance, \cite{hal} and
references therein), we know that there is a unique local 
solution $(x(t), y(t), v(t), z(t))$ to system \eqref{sys} 
in $[0,t_m)$. From system \eqref{sys}, we have the following:
$$
\displaystyle x(t) = e^{-\displaystyle \int_0^t (d+\beta v(\xi))d\xi}\left(x(0)
+\int_0^t \lambda e^{\displaystyle \int_0^\eta (d+\beta v(\xi))d\xi}
d\eta\right),
$$
$$
\displaystyle y(t) = e^{-\displaystyle \int_0^t (a+pz(\xi))d\xi}\left( y(0)+
\int_0^t \beta v(\eta-\tau)x(\eta-\tau) e^{\displaystyle \int_0^\eta
(a+pz(\xi))d\xi} d\eta \right),
$$
$$
\displaystyle v(t) = e^{-\mu t}\left( v(0)+ \int_0^t aN y(\eta)
e^{\mu \eta} d\eta\right),
$$
and
$$
\displaystyle z(t) = z(0)e^{\displaystyle \int_0^t (cy(\xi)-h)d\xi}.
$$
This shows the positivity of solutions in $t \in [0,t_m)$.
Next, for the boundedness of the solutions, we consider 
the following function:
$$
F(t)= aN x(t) + aN y(t+\tau) + \frac{a}{2} v(t+\tau).
$$
This leads to
\begin{align*}
\frac{dF(t)}{dt} &= aN \left(\lambda -d
x(t)-\beta v(t)x(t)\right)\\
& \quad + aN \left( \beta v(t)x(t) - ay(t+\tau) -p
y(t+\tau)z(t+\tau)\right)\\
& \quad + \frac{a}{2}\left( aN y(t+\tau) -\mu v(t+\tau)  \right),
\end{align*}
from which we have
$$
\frac{dF(t)}{dt} \le \lambda aN  - aN d  x(t) - \frac{a^{2}N}{2}
y(t+\tau)- \frac{a \mu}{2} v(t+\tau).
$$
If we set $\displaystyle \varrho = \min\left(d,\frac{a}{2},\mu\right)$, 
then we have
$$
\frac{dF(t)}{dt} \le \lambda aN - \varrho F(t).
$$
This proves, via Gronwall's lemma, that $F(t)$ is bounded, 
and so are the functions $x(t)$, $y(t)$ and $v(t)$.
Now, we prove the boundedness of $z(t)$. From the last
equation of \eqref{sys}, we have
$$
\dot{z}(t) + hz(t) = cx(t)y(t)z(t).
$$
Moreover, from the second equation of \eqref{sys}, it follows that
$$
\dot{z}(t) + hz(t) = \frac{c}{p}x(t)\left(\beta
x(t-\tau)v(t-\tau)- ay(t)- \dot{y}(t) \right).
$$
Thus, by integrating over time, we have
$$
z(t) = z(0) e^{-ht} + \int_{0}^t \frac{c}{p}x(s)\left(  \beta
x(s-\tau)v(s-\tau)- ay(s) - \dot{y}(s)\right) e^{h(s-t)}ds.
$$
From the boundedness of $x$, $y$ and $v$, and by using integration
by parts, it follows the boundedness of $z(t)$. Therefore, every 
local solution can be prolonged up to any time $t_m>0$, 
which means that the solution exists globally.
\end{proof}

% --------------------------------------

\subsection{The linearized problem around the steady-state solution}
\label{sec:2:2}

It is straightforward to establish that model
\eqref{sys} has one disease-free equilibrium given by
$$
E_{f}=\left(\frac{\lambda}{d},0,0,0\right)
$$
and two endemic equilibrium points given as follows:
$$
E_{1} = \left(\frac{ \mu}{N \beta},\frac{ \lambda \beta N-d \mu}{aN \beta},
\frac{\lambda \beta N-d \mu}{\mu \beta},0\right)
$$
and
$$
E_{2} = \left(\frac{\lambda \mu c- \beta aN h}{d \mu c}, 
\frac{dh \mu}{\lambda \mu c- \beta aN h}, \frac{dh aN}{\lambda \mu c
- \beta aN h},\frac{\beta aN}{\mu p}\left(\frac{\lambda \mu c
- \beta aN h}{d \mu c}\right)-\frac{a}{p}\right).
$$
Consider now the following transformation:
$$
X(t)= x(t)- \bar{x},
\quad Y(t)= y(t)- \bar{y}, 
\quad V(t)= v(t)-  \bar{v}, 
\quad Z(t)= z(t)- \bar{z},
$$
where $\left(\bar{x},\bar{y},\bar{v},\bar{z}\right)$ denotes 
any equilibrium point $E_{f}$, $E_{1}$ or $E_{2}$. The linearized 
system of the previous model \eqref{sys} is of form
\begin{equation}
\label{sys1}
\begin{cases}
\displaystyle \dot{X}(t) = \left(-d- \beta \bar{v}\right)X(t) 
- \beta \bar{x} V(t), \\
\displaystyle \dot{Y}(t) = \left(-a-p\bar{z}\right)Y(t)
+ \beta \bar{v} X(t- \tau )+ \beta \bar{x} V(t- \tau) -p \bar{y} Z(t),\\
\displaystyle \dot{V}(t) = aN Y(t)- \mu V(t) , \\
\displaystyle \dot{Z}(t) = c \bar{y} \bar{z} X(t) + c\bar{x} \bar{z} Y(t) 
+ (c \bar{x} \bar{y} -h) Z(t).
\end{cases}
\end{equation}
System \eqref{sys1} can be written in matrix form as follows:
\begin{equation*}
\frac{d}{dt} 
\left(
\begin{array}{c}
X(t) \\
Y(t) \\
V(t) \\
Z(t)
\end{array}
\right)
= A_{1} 
\left(
\begin{array}{c}
X(t) \\
Y(t) \\
V(t) \\
Z(t)
\end{array}
\right) 
+A_{2} 
\left(
\begin{array}{c}
X(t- \tau) \\
Y(t- \tau) \\
V(t- \tau) \\
Z(t- \tau) 
\end{array}
\right),
\end{equation*}
where $A_{1}$ and $A_{2}$ are the two matrices given by
\begin{equation*}
A_{1}= 
\left(
\begin{array}{cccc}
-d- \beta \bar{v} & 0 & - \beta \bar{x} & 0 \\
0 & -a-p \bar{z} & 0 & -p \bar{y} \\
0 & aN & -\mu & 0 \\
c \bar{y} \bar{z} & c \bar{x} \bar{z} & 0 & c \bar{x} \bar{y} -h \\
\end{array}
\right)
\end{equation*}
and
\begin{equation*}
A_{2} = 
\left(
\begin{array}{cccc}
0 & 0 & 0 & 0 \\
\beta \bar{v} & 0 & \beta \bar{x} & 0 \\
0 & 0 & 0 & 0 \\
0 & 0 & 0 & 0 \\
\end{array}
\right).
\end{equation*}

% --------------------------------------

\subsection{Stability of the disease-free equilibrium}
\label{sec:2:3}

We begin by studying the stability of the disease-free 
equilibrium $E_f$. The following result holds.

\begin{theorem}
\label{x}
The local stability of the disease-free equilibrium $E_f$ depends
on the value of $N \beta \lambda - d  \mu$. Precisely, we have:
\begin{enumerate}
\item if $N \beta \lambda - d  \mu < 0 $, then the disease-free
equilibrium $E_{f}$ is locally asymptotically stable 
for any time delay $\tau \geq 0$;

\item if $N \beta \lambda - d  \mu > 0$, then the equilibrium $E_{f}$ 
is unstable for any time delay $\tau \geq 0$.
\end{enumerate}
\end{theorem}

\begin{proof}
The characteristic equation of system \eqref{sys} is given by
\begin{equation}
\label{eq1}
\Delta ( \zeta ) = \det( \zeta Id - A_{1} - e^{- \zeta \tau} A_{2})=0.
\end{equation}
Thus, at the disease-free equilibrium, 
the characteristic equation takes the form
\begin{equation}
\label{eq2}
(\zeta +d)( \zeta +h)\left[\zeta ^{2} +( \mu + a) \zeta 
+a \mu\left(1-\frac{N \beta \lambda }{d \mu} e^{- \zeta \tau}\right)\right]=0.
\end{equation}
If we assume $\tau = 0$, then equation \eqref{eq2} becomes
\begin{equation}
\label{eq3}
(\zeta +d)(\zeta +h)\left[\zeta ^{2} +(\mu + a) \zeta 
+a \mu\left(1-\frac{N \beta \lambda }{d \mu}\right)\right]=0.
\end{equation}
The four roots of \eqref{eq3} are: 
\begin{equation*}
\begin{split}
\zeta_{1}&=-d,\\
\zeta_{2}&= -h,\\
\zeta_{3}&=\frac{-( \mu +a) 
- \sqrt{( \mu+a)^{2}-4a \mu \left(1-\frac{N \beta \lambda}{d \mu}\right)}}{2},\\
\zeta_{4}&= \frac{-( \mu +a) + \sqrt{( \mu +a)^{2}-4a \mu
(1-\frac{N \beta \lambda}{d \mu})}}{2}. 
\end{split}
\end{equation*}
It is clear that $\zeta_{1}$, $\zeta_{2}$ and $\zeta_{3}$ have negative real parts,
while $\zeta_{4}$ has negative real part if $N \beta \lambda - d
\mu < 0$. Suppose that $\tau>0$. To prove the stability of $E_{f}$, 
we use Rouch\'e's theorem. For that, we need to prove that the roots
of the characteristic equation \eqref{eq2} cannot have pure
imaginary roots, that is, cannot cross the imaginary axis. Suppose
the contrary. Let $\zeta = \omega i$ with $\omega >0$ a purely
imaginary root of \eqref{eq2}. Then, $(\omega i) ^{2} +(\mu + a) 
(\omega i) +a \mu\left(1- \frac{N \beta \lambda }{d \mu} e^{-i\omega\tau}\right)=0$,
that is, $- \omega ^{2} +( \mu + a) \omega i 
+a \mu\left(1- \frac{N \beta \lambda }{d \mu} e^{-i \omega \tau}\right)=0$.
By using Euler's formula $e^{-i \omega \tau}= \cos(\omega \tau) 
-i \sin(\omega \tau)$ and by separating the real imaginary parts, 
we have
\begin{equation*}
\begin{cases}
\displaystyle -\omega^{2}+ a \mu = \frac{aN \beta \lambda}{d} \cos(\omega \tau)\\[0.3cm]
\displaystyle (a+ \mu) \omega = - \frac{aN \beta \lambda}{d} \sin(\omega \tau).
\end{cases}
\end{equation*}
Adding the squares in the two equations, one obtains
$$
\omega^{4} +\left(a^{2}+ \mu^{2}\right) \omega ^{2}+ a^{2} 
\mu ^{2}\left(1-\left(\frac{N \beta \lambda}{d \mu}\right)^{2}\right)=0.
$$
Let $X= \omega^{2}$. We have
$$
X^{2} + \left(a^{2}+ \mu^{2}\right) X
+ a^{2} \mu^{2}\left(1 - \left(\frac{N \beta\lambda}{d \mu}\right)^{2}\right)=0,
$$
which has no positive solution when $N \beta \lambda - d \mu < 0$.
Therefore, there is no root $\zeta= i \omega$ with $\omega \geq 0$ 
for \eqref{eq2}, implying that the root of \eqref{eq2} cannot
intersect the pure imaginary axis. Therefore, all roots of
\eqref{eq2} have negative real parts $N \beta \lambda - d \mu 
< 0$ and the disease-free equilibrium, $E_{f}$, is locally
asymptotically stable if $N \beta \lambda - d \mu < 0$.
In addition, it is easy to show that \eqref{eq2} has a real positive
root when $N \beta \lambda - d \mu > 0$. Indeed, let us put
$$
f(\zeta)= \zeta ^{2} +\left(\mu + a\right) \zeta 
+a \mu\left(1-\frac{N \beta \lambda }{d \mu} e^{- \zeta \tau}\right).
$$ 
Then, $f(0)=a\mu\left(1- \frac{N \beta \lambda }{d \mu}\right)>0$ 
and $\lim_{\zeta\rightarrow +\infty} f(\zeta)= +\infty$. 
Consequently, $f$ has a positive real root and the disease-free 
equilibrium is unstable.
\end{proof}

% --------------------------------------

\subsection{Stability of the endemic equilibria}
\label{sec:2:4}

We start by studying the local stability of the
infected-equilibrium $E_{1}$ for any time delay $\tau$.

\begin{theorem}
The local stability of the disease-free equilibrium $E_1$ depends
on the value of 
$$
\beta N(\mu c \lambda - \beta h aN)-\mu^{2} cd.
$$
Precisely, we have:
\begin{enumerate}
\item if $\beta N(\mu c \lambda - \beta h aN)-\mu^{2} cd <0$,
then $E_{1}$ is locally asymptotically stable for any positive 
time delay $\tau$;

\item if $\beta N(\mu c \lambda - \beta h aN)-\mu^{2} cd > 0$,
then $E_{1}$ is unstable for any positive time delay $\tau$.
\end{enumerate}
\end{theorem}

\begin{proof}
Let $\lambda \beta N - d \mu >0 $ and $\mu c \lambda -\beta h aN >0$. 
The characteristic equation \eqref{eq1} at $E_{1}$ is given by
\begin{equation}
\label{eq4}
\left(\zeta -c \bar{x} \bar{y}+h\right)\left(
\zeta^{3} + A \zeta^{2} +B \zeta +C 
-e^{-\zeta \tau}\left(g_{1} \zeta + g_{2}\right)\right)=0,
\end{equation}
where 
\begin{equation*}
\begin{split}
A&=d+ \mu +a + \beta \bar{v},\\
B&= \mu d +ad + a \mu + \mu \beta \bar{v},\\
C&= a \mu (d+ \beta \bar{v}) - \beta aN d \bar{x},\\
g_{1} &= \beta aN \bar{x},\\
g_{2} &= \beta aN d \bar{x}.
\end{split}
\end{equation*}
Note that 
\begin{equation}
\label{eq:zeta:sol}
\zeta = \frac{\beta N\left(\mu c \lambda - \beta h aN\right)
-\mu^{2} cd}{aN^{2}\beta^{2}}
\end{equation}
is a solution of \eqref{eq4}. If 
$\beta N(\mu c \lambda - \beta h aN)-\mu^{2} cd <0$, then \eqref{eq:zeta:sol}
is a real negative root of the characteristic equation \eqref{eq4}, 
and we just need to analyze equation
\begin{equation}
\label{eq5}
\zeta^{3} + A \zeta^{2} +B \zeta +C - e^{-\zeta \tau}(g_{1} 
\zeta + g_{2})=0.
\end{equation}
Consider now $\tau =0$. From equation \eqref{eq5}, we have
\begin{equation}
\label{eq6}
\zeta^{3} + P \zeta^{2} +Q \zeta +R=0,
\end{equation}
where
\begin{equation*}
\begin{split}
P&= d+ \mu +a + \beta \bar{v},\\
Q&= \mu d +ad + a \mu + \mu \beta \bar{v} - aN \beta \bar{x},\\
R&= a \mu (d+ \beta \bar{v}) -2 \beta  aN d \bar{x}.
\end{split}
\end{equation*}
Because $\beta N \lambda - d \mu >0 $, from the Routh--Hurwitz stability criterion, 
it follows that all roots of \eqref{eq6} have negative real part. Thus, 
$E_{1}$ is locally asymptotically stable for $\tau =0$.
Let $\tau >0$. Suppose that \eqref{eq4} has pure imaginary roots
$\zeta = \omega i$ with $\omega >0$. If we replace $\zeta$ in
\eqref{eq5} by $\zeta = \omega i$, and separate the real and
imaginary parts, then we obtain
\begin{equation*}
\begin{cases}
\displaystyle - A \omega^{2} + C 
= g_{2} \cos(\omega \tau)+ g_{1} \omega \sin(\omega \tau),\\
- \omega^{3} + B \omega 
= g_{1} \omega \cos(\omega \tau) - g_{2} \sin(\omega \tau).
\end{cases}
\end{equation*}
By adding up the squares of the two equations, and by using the
fundamental trigonometric formula, we obtain that
$$
\omega^{6} + (A^{2} -2 B ) \omega^{4} + ( B^{2} -2AC- g_{1}^{2} )
\omega^{2} +C^{2} -g_{2}^{2} =0.
$$
Letting $X= \omega^{2}$, yields
\begin{equation*}
F(X)= X^{3} + (A^{2} -2B ) X^{2} + ( B^{2} -2AC- g_{1}^{2} ) X 
+C^{2} -g_{2}^{2} =0.
\end{equation*}
We have $F(0)= \lambda^{2} \beta^{2}a^{2}N^{2}-a^{2}
\mu^{2}d^{2}>0$ and $\lim_{\zeta\rightarrow +\infty} f(\zeta)
= +\infty$. Hence, \eqref{eq5} has no positive solution because
$\lambda \beta N - d \mu >0 $. Therefore, there is no root $ \zeta
= \omega i$ with $ \omega >0$ for \eqref{eq5}, implying that the
root of \eqref{eq5} cannot cross the purely imaginary axis. Thus,
all roots of \eqref{eq4} have negative real parts. Then, 
$E_{1}$ is locally asymptotically stable when 
$\beta N(\mu c \lambda-\beta h aN)-\mu^{2} cd <0$.
\end{proof}

For the second endemic equilibrium point $E_2$, 
the following result holds.

\begin{theorem} 
\label{y}
Assume that $\lambda \mu c- \beta aN h >0$. If $ \beta N(\lambda
\mu c - \beta h aN)-\mu^{2} cd >0 $, then the infected equilibrium
$E_{2}$ is locally asymptotically stable for $\tau = 0$.
\end{theorem}

\begin{proof}
Let $ \beta N(\lambda \mu c - \beta h aN)-\mu^{2} cd >0 $. The
characteristic equation \eqref{eq1} at $E_{2}$ is given by
\begin{multline}
\label{eq8}
\zeta^{4}+ A \zeta^{3}+ B \zeta^{2}+ C \zeta +D+[-\beta
aN\bar{x}
\zeta^{2}+(c \beta aN \bar{y}\bar{x}^{2}
-\beta h aN \bar{x}-\beta aN d \bar{x}) \zeta \\
+c \beta aN d\bar{y} \bar{x}^{2}-\beta hd aN \bar{x}]e^{-\zeta \tau}=0,
\end{multline}
where
\begin{equation*}
\begin{aligned}
A&= \mu + a + d + p\bar{z}+ \beta \bar{v},& \\ 
B&= a \mu + \mu d + ad +p \mu \bar{z}+ pd\bar{z}-ph \bar{z}
+ \beta \mu \bar{v}+a \beta \bar{v}+p \beta \bar{z}\bar{v} ,&\\
C&= ad \mu + p \mu h \bar{z}+ phd \bar{z}+ p \mu d \bar{z}
+ a \mu \beta \bar{v}+ ph \beta \bar{z}\bar{v}+ p \mu \beta \bar{z}\bar{v},&\\
D&= p \mu hd \bar{z}+ p \mu h \beta \bar{z}\bar{v}
- aN pc \beta\bar{x}\bar{z}\bar{y}^{2}.&
\end{aligned}
\end{equation*}
If $\tau=0$, then the characteristic equation \eqref{eq8} becomes
\begin{equation*}
\zeta^{4}+ E \zeta^{3}+ F \zeta^{2}+ G \zeta + H=0,
\end{equation*}
where
\begin{equation*}
\begin{aligned}
E&= \mu + a + d + p\bar{z}+ \beta \bar{v},\\ 
F&= a \mu + \mu d + ad +p \mu \bar{z}+ pd\bar{z}-ph \bar{z}+ \beta \mu \bar{v}
+a \beta \bar{v}+p \beta \bar{z}\bar{v}- \beta aN \bar{x},\\
G&= ad \mu + p \mu h \bar{z}+ phd \bar{z}+ p \mu d \bar{z}+ a \mu \beta \bar{v}
+ ph \beta \bar{z}\bar{v}+ p \mu \beta \bar{z}\bar{v}- \beta aN d \bar{x},\\
H&= p \mu hd \bar{z}+ p \mu h \beta \bar{z}\bar{v}- aN \beta ph \bar{y}\bar{z}.
\end{aligned}
\end{equation*}
Thus, $E>0$, $F>0$, $G>0$, $H>0$ and $FG-EH>0$, 
whenever $ \beta N(\lambda \mu c - \beta h aN)-\mu^{2} cd >0$.
\end{proof}

Let $\tau>0$. Suppose that \eqref{eq8} has pure imaginary roots
$\zeta= \omega i$. Replacing $\zeta$ in \eqref{eq8} by $\omega i$,
and separating the real and imaginary parts, we obtain that
\begin{equation*}
\begin{cases} 
\displaystyle \omega^{4}-I\omega^{2}+J = K \omega^{2} + L \omega + M\\
-O \omega^{3}+ N \omega = P \omega^{2}+ Q \omega +R
\end{cases}
\end{equation*}
with
\begin{equation*}
\begin{aligned}
I&= a \mu + \mu d + ad +p \mu \bar{z}+ pd\bar{z}-ph \bar{z}
+\beta\mu \bar{v}+a \beta \bar{v}+p \beta \bar{z}\bar{v},\\
J&= p \mu hd \bar{z}+ p \mu h \beta \bar{z}\bar{v}
- aN pc \beta\bar{x}\bar{z}\bar{y}^{2},\\
K&= -\beta aN \bar{x}\cos(\omega \tau),\\
L&= -c \beta aN \bar{y} \bar{x}^{2}\sin(\omega \tau)
+ \beta h aN \bar{x} \sin(\omega \tau)+ \beta aN d \bar{x}  \sin(\omega \tau),\\
M&= -c \beta aN d\bar{y}\bar{x}^{2}\cos(\omega \tau) +\beta hd aN\bar{x}\cos(\omega \tau),\\
N&= ad \mu + p \mu h \bar{z}+ phd \bar{z}+ p \mu d \bar{z}
+ a \mu \beta \bar{v}+ ph \beta \bar{z}\bar{v}+ p \mu \beta \bar{z}\bar{v},\\
O&= \mu + a + d + p\bar{z}+ \beta \bar{v},\\
P&= \beta aN \bar{x}\sin(\omega \tau),\\
Q&= -c \beta aN \bar{y}\bar{x}^{2}\cos(\omega \tau) 
+ \beta h aN \bar{x}\cos(\omega \tau)+ \beta aN d \bar{x} \cos(\omega \tau),\\
R&= c \beta aN d \bar{y}\bar{x}^{2}\sin(\omega \tau)-\beta hd aN \bar{x}\sin(\omega \tau).
\end{aligned}
\end{equation*}
By adding up the squares of both equations, and using the
fundamental trigonometric formula, we obtain that
\begin{equation}
\label{eq9}
\omega^{8}+ S \omega^{6}+ T \omega^{4}+ U \omega^{2} +V=0,
\end{equation}
where
\begin{equation*}
\begin{aligned}
S&= O^{2}-2I,& \\ 
T&= 2J+I^{2}-2NO- \beta^{2}a^{2}N^{2}\bar{x}^{2},&\\
U&=2IJ+N^{2}-c^{2}\beta^{2}a^{2}N^{2}\bar{y}^{2}\bar{x}^{4}
+2c\beta^{2}a^{2}N^{2}h\bar{y}\bar{x}^{3}\\
&\quad -\beta^{2}a^{2}N^{2}h^{2}\bar{x}^{2}
+2c\beta^{2}a^{2}N^{2}d\bar{y}\bar{x}^{3}
-2\beta^{2}a^{2}N^{2}hd\bar{x}^{2}\\
&\quad -\beta^{2}a^{2}N^{2}d^{2}\bar{x}^{2}
-2c \beta^{2}a^{2}N^{2}d \bar{y}\bar{x}^{3}
+2 \beta^{2}a^{2}N^{2}hd \bar{x}^{2},&\\
V&= J^{2}+2c\beta^{2}a^{2}N^{2}d^{2}h\bar{y}\bar{x}^{3}
-\beta^{2}h^{2}d^{2}a^{2}N^{2}\bar{x}^{2}.&
\end{aligned}
\end{equation*}
Equation \eqref{eq9} admits at least two pure imaginary roots. 
Indeed, let $\lambda=1$, $d=\frac{1}{10}$, $\beta=\frac{1}{2}$,
$a=\frac{1}{5}$, $p=1$, $c=\frac{1}{10}$, $h=\frac{1}{10}$,
$N=1500$, $\mu=1$. Then, equation \eqref{eq9} is given by
$$
\omega^8+\frac{3262009}{360000} \omega^6-\frac{419609}{4500000} \omega^4
+\frac{1060237}{300000000} \omega^2+\frac{313}{2000000} = 0.
$$
This equation admits four pure imaginary roots; due to length 
of their writing space, we give here their approximated values:
$$
0.1550207983i,\quad -0.1550207983i, \quad 3.008467478i\ 
\text{ and }\ -3.008467478i.
$$
Therefore, from Rouch\'e's theorem, we cannot conclude anything
about the stability of $E_{2}$. Numerically, however, we can show 
that the endemic equilibrium $E_{2}$ is locally asymptotically stable 
for certain values of $\tau$. For example, let $\tau=10$, $\lambda=1$,
$d=\frac{1}{10}$, $\beta=0.00025$, $p=0.001$, $a=0.2$, $c=0.03$,
$N=1500$, and $\mu=3$. In this case, it is easy to show analytically 
that the characteristic equation \eqref{eq8} is given by
$f(\zeta)=0$ with
$$
f(\zeta)=\zeta^{4}+\frac{1997}{600} \zeta^{3}+\frac{121}{120} \zeta^{2}
+\frac{401}{5000} \zeta +\frac{1}{2000}-\frac{5}{8}e^{-\zeta \tau} \zeta^{2}
-\frac{1}{16}e^{-\zeta \tau} \zeta.
$$
Thus, $f(0)= \frac{1}{2000}$ and the derivative is always positive for
$\tau \geq 0$. Therefore, $f(\zeta)$ does not have nonnegative
real roots. Analogously, we can show numerically that $E_{2}$ is locally
asymptotically stable for some other positive values of the time delay
$\tau$. A general result remains, however, an open question.

% -------------------------------------- 

\section{Optimal control}
\label{sec:3}

In this section, we study an optimal control problem 
associated with the delayed HIV model 
with CTL immune response \eqref{sys}.

% ----------------

\subsection{The optimization problem}
\label{sec:3:1}

We suggest the following delayed control system 
with two control variables $u_1$ and $u_2$:
\begin{equation}
\label{sy1}
\begin{cases}
\displaystyle \frac{dx(t)}{dt} 
= \lambda -dx(t)- \beta(1-u_{1}(t)) x(t)v(t), \\[0.3cm]
\displaystyle \frac{dy(t)}{dt} 
= \beta (1-u_{1}(t)) x(t-\tau)v(t-\tau)- ay(t)- py(t)z(t),\\[0.3cm]
\displaystyle \frac{dv(t)}{dt} 
= aN(1-u_{2}(t))y(t)- \mu v(t) , \\[0.3cm]
\displaystyle \frac{d z(t)}{dt} 
= cx(t)y(t)z(t) - hz(t),
\end{cases}
\end{equation}
where the controls belong to the control set $U$ defined by
\begin{equation*}
U=\left\{(u_{1},u_{2}) : u_{i} \text{ is measurable}, 
\quad 0\leq u_{i}(t)\leq 1, \quad t\in [0,t_{f}], \quad i=1,2\right\}.
\end{equation*}
Here, $u_1$ represents the efficiency of drug therapy in blocking
new infections, so that the infection rate in presence of drug is
$(1 - u_1)$; while $u_2$ stands for the efficiency of drug therapy
in inhibiting viral production, such that the virion production
rate under therapy is $(1 - u_2)$. The optimization problem under 
consideration is to maximize the objective functional
\begin{equation}
\label{sy2}
\begin{aligned}
\mathcal{J}(u_{1},u_{2})=\int^{t_{f}}_{0}\left\{x(t)+z(t)
-\left[\frac{A_{1}}{2}u_{1}^{2}(t)
+\frac{A_{2}}{2}u_{2}^{2}(t)\right]\right\}dt,
\end{aligned}
\end{equation}
where $t_{f}$ is the time period of treatment and the positive
constants $A_{1}$ and $A_{2}$ stand for the benefits and costs of
the introduced treatment, subject to the control system \eqref{sy1}. 
The two control functions, $u_{1}(\cdot)$ and $u_{2}(\cdot)$, 
are assumed to be bounded and Lebesgue integrable. Summarizing,
the optimal control problem under study consists to find
$u_{1}^{*}$ and $u_{2}^{*}$ such that
\begin{equation}
\label{sy3}
\begin{aligned}
J(u_{1}^{*},u_{2}^{*})&=\max\{\mathcal{J}(u_{1},u_{2}) : (u_{1},u_{2})\in U\}\\
&\text{subject to } \eqref{sy1}, \eqref{2}, \eqref{3}.
\end{aligned}
\end{equation}
Pontryagin's minimum principle \cite{Gollmann} provides necessary optimality
conditions for such optimal control problem with delays. Roughly speaking,
this principle reduces \eqref{sy3} to a problem of maximizing an Hamiltonian $H$ 
pointwisely with respect to $u_{1}$ and $u_{2}$. In our case the Hamiltonian 
is given by
$$
H(x,y,v,z,x_{\tau},v_{\tau},u_{1},u_{2},\psi_1,\psi_2,\psi_3,\psi_4)
=\frac{A_{1}}{2}u_{1}^{2}+\frac{A_{2}}{2}u_{2}^{2}-x-z
+\displaystyle\sum_{i=0}^{4}\psi_{i}f_{i}(x,y,v,z,x_{\tau},v_{\tau},u_{1},u_{2})
$$
with
\begin{equation*}
\left\{
\begin{aligned}
f_{1}&= \lambda-dx- \beta(1-u_{1})xv, \\
f_{2}&= \beta(1-u_{1})x_{\tau}v_{\tau}-ay-pyz, \\
f_{3}&= aN(1-u_{2})y- \mu v, \\
f_{4}&= cxyz-hz.
\end{aligned}
\right.
\end{equation*}
By applying Pontryagin's minimum principle \cite{Gollmann}, 
we obtain the following result.

\begin{theorem} 
For any initial conditions \eqref{2}--\eqref{3}, 
the system \eqref{sy1} has a unique solution. 
This solution is nonnegative and bounded for all $t \ge 0$. 
In addition, if $u_{1}^{*}$ and $u_{2}^{*}$
are optimal controls and $x^{*}$, $y^{*}$, $v^{*}$ and $z^{*}$ 
corresponding solutions of the state system \eqref{sy1}, 
then there exists adjoint variables $\psi_{1}$, $\psi_{2}$, 
$\psi_{3}$ and $\psi_{4}$ satisfying the adjoint equations
\begin{equation*}
\begin{cases}
\psi'_{1}(t)
=1+\psi_{1}(t)\big[d+\big(1-u_{1}^{*}(t)\big)\beta v^{*}(t)\big] 
- \psi_{4}(t)cy^{*}(t)z^{*}(t)\\
\qquad\quad +\chi_{[0,t_{f}-\tau]}(t)\psi_{2}\big(t+\tau\big)
\big(u_{1}^{*}\big(t+\tau\big)-1\big)\beta v^{*}(t),\\
\psi'_{2}(t)=\psi_{2}(t)a-\psi_{3}(t)\big(1-u^{*}_{2}(t)\big)aN
- \psi_{4}(t)cx^{*}(t)z^{*}(t)+ \psi_{2}(t)pz^{*}(t),\\
\psi'_{3}(t)=\psi_{1}(t)\big[\beta(1-u_{1}^{*}(t))x^{*}(t)\big]+\psi_{3}(t)\mu 
+\chi_{[0,t_{f}-\tau]}(t)\psi_{2}(t+\tau)\left[\beta (u^{*}_{1}(t+\tau)-1)x^{*}(t)\right],\\
\psi'_{4}(t)=1+\psi_{2}(t)py^{*}(t)+\psi_{4}(t)\left[h-cx^{*}(t)y^{*}(t)\right]
\end{cases}
\end{equation*}
with transversality conditions
\begin{equation*}
\psi_{i}(t_{f})=0, \quad i=1,\ldots,4.
\end{equation*}
Moreover, the optimal controls satisfy
\begin{equation}
\label{opt}
\begin{aligned}
u^{*}_{1}(t)= &\min\bigg(1,\max\bigg(0,\frac{\beta}{A_{1}}\bigg[\psi_{2}(t)
v^{*}(t-\tau) x^{*}(t-\tau)-\psi_{1}(t)v^{*}(t)x^{*}(t)\bigg]\bigg)\bigg),\\
u^{*}_{2}(t)= &
\min\bigg(1,\max\bigg(0,\frac{1}{A_{2}}\psi_{3}(t)aNy^{*}(t)\bigg)\bigg).
\end{aligned}
\end{equation}
\end{theorem}

\begin{proof}
The proof of positivity and boundedness of solutions is
similar to the one of Theorem~\ref{thm1}. It is enough to
use the fact that $u_i(t) \in U$, $i=1,2$, which means that
$\|u_i (t)\|_{L^{\infty}} \le 1$. For the rest of the proof, 
we remark that the adjoint equations and transversality conditions 
are obtained by using the Pontryagin minimum principle with delays 
of \cite{Gollmann}, from which
\begin{equation*}
\begin{cases}
\psi'_{1}(t)=-\frac{\partial H}{\partial x}(t)
-\chi_{[0,t_{f}-\tau]}(t)\frac{\partial H}{\partial
x_{\tau}}(t+\tau),\qquad &\psi_{1}(t_{f})=0,\\
\psi'_{2}(t)=-\frac{\partial H}{\partial y}(t),
\qquad &\psi_{2}(t_{f})=0,\\
\psi'_{3}(t)=-\frac{\partial H}{\partial v}(t)-
\chi_{[0,t_{f}-\tau]}(t)\frac{\partial H}{\partial
v_{\tau}}(t+\tau),\qquad &\psi_{3}(t_{f})=0,\\
\psi'_{4}(t)=-\frac{\partial H}{\partial z}(t),
\qquad &\psi_{4}(t_{f})=0.
\end{cases}
\end{equation*}
From the optimality conditions,
\begin{equation*}
\frac{\partial H}{\partial u_{1}}(t)=0,
\qquad\frac{\partial H}{\partial u_{2}}(t)=0,
\end{equation*}
that is,
\begin{gather*}
A_{1}u_{1}(t)+\beta\psi_{1}(t)v(t)x(t)-\beta
\psi_{2}(t)v(t-\tau) x(t-\tau)= 0,\\
A_{2}u_{2}(t)-aN \psi_{3}(t)y(t)= 0.
\end{gather*}
Taking into account the bounds in $U$ for the two controls, 
one obtains $u_{1}^{*}$ and $u_{2}^{*}$ in form \eqref{opt}.
\end{proof}

% --------------------------------------

\subsection{Existence of an optimal control pair}
\label{sec:3:2}

The existence of the optimal control pair can be directly obtained
using the results in \cite{Fleming,Lukes}. More precisely, we have
the following theorem.

 \begin{theorem}  
There exists an optimal control
pair $(u_{1}^{*},u_{2}^{*})\in U$ solution of \eqref{sy3}.
\end{theorem}

\begin{proof}
To use the existence result in \cite{Fleming}, 
we first need to check the following properties:
\begin{enumerate}
\item [$(P_1)$] the set of controls and corresponding state variables is nonempty;

\item [$(P_2)$] the control set $U$ is convex and closed;

\item [$(P_3)$] the right-hand side of the state system 
is bounded by a linear function in the
state and control variables;

\item [$(P_4)$] the integrand of the
objective functional is concave on $U$;

\item [$(P_5)$] there exist constants $c_{1}, c_{2} > 0$
and $\beta > 1$ such that the integrand 
$$
L(x,z,u_{1},u_{2}) = x + z 
- \left( \frac{A_{1}}{2}u_{1}^{2} +\frac{A_{2}}{2}u_{2}^{2}\right)
$$ 
of the objective functional \eqref{sy2} satisfies
\begin{equation*}
L(x,z,u_{1},u_{2}) \leq c_{2}-c_{1}(\mid u_{1}\mid^{2} 
+ \mid u_{2}\mid^{2})^{^{\frac{\beta}{2}}}.
\end{equation*}
\end{enumerate}
Using the result in \cite{Lukes}, 
we obtain existence of solutions of system
\eqref{sy1}, which gives condition $(P_1)$. The control set is
convex and closed by definition, which gives condition $(P_2)$.
Since our state system is bilinear in $u_{1}$ and $u_{2}$, the 
right-hand side of system \eqref{sy1} satisfies condition $(P_3)$, 
using the boundedness of solutions. Note that the integrand of our
objective functional is concave. Also, we have the last needed
condition:
\begin{equation*}
L(x,z,u_{1},u_{2}) \leq c_{2}
-c_{1}\left(\mid u_{1}\mid^{2} 
+ \mid u_{2}\mid^{2}\right),
\end{equation*}
where $c_{2}$ depends on the upper bound on $x$ and $z$, and
$c_{1}>0$ since $A_{1}>0$ and $A_{2}>0$. We conclude that there
exists an optimal control pair $(u_{1}^{*},u_{2}^{*})\in U$ such that
\begin{equation*}
J(u_{1}^{*},u_{2}^{*})
= \displaystyle\max_{(u_{1},u_{2})\in U} \mathcal{J}(u_{1},u_{2})
\end{equation*}
subject to \eqref{sy1}, \eqref{2} and \eqref{3}. 
The proof is complete.
\end{proof}

% --------------------------------------

\subsection{The optimality system}
\label{sec:3:3}

The optimality system consists of the state system coupled 
with the adjoint equations, the initial conditions, 
transversality conditions, and the
characterization of optimal controls \eqref{opt}. Precisely,
if we substitute the expressions of $u_{1}^{*}$ and $u_{2}^{*}$ 
in \eqref{sy1}, then we obtain the following optimality system:
\begin{equation*}
\begin{cases}
\displaystyle \frac{dx^{*}(t)}{dt}= \lambda- dx^{*}(t)
-\beta(1-u^{*}_{1}(t)) x^{*}(t) v^{*}(t),\\[0.3cm]
\displaystyle \frac{dy^{*}(t)}{dt}= \beta(1-u^{*}_{1}(t))
x^{*}(t-\tau) v^{*}(t-\tau)
-ay^{*}(t)-py^{*}(t)z^{*}(t),\\[0.3cm]
\displaystyle \frac{dv^{*}(t)}{dt}= aN(1-u_{2}^{*}(t))y^{*}(t)
-\mu v^{*}(t),\\[0.3cm]
\displaystyle \frac{dz^{*}(t)}{dt}
=cx^{*}(t)y^{*}(t)z^{*}(t)-hz^{*}(t),\\[0.3cm]
\displaystyle \frac{d\psi_{1}(t)}{dt}
=1+\psi_{1}(t)\big[d+\big(1-u_{1}^{*}(t)\big)\beta v^{*}(t)\big] 
- \psi_{4}(t)cy^{*}(t)z^{*}(t)\\[0.3cm]
\qquad\qquad +\chi_{[0,t_{f}-\tau]}(t)\psi_{2}\big(t
+\tau\big)\big(u_{1}^{*}\big(t+\tau\big)-1\big)\beta v^{*}(t),\\[0.3cm]
\displaystyle \frac{d\psi_{2}(t)}{dt}
=\psi_{2}(t)a-\psi_{3}(t)\big(1-u^{*}_{2}(t)\big)
aN- \psi_{4}(t)cx^{*}(t)z^{*}(t)+ \psi_{2}(t)pz^{*}(t),\\[0.3cm]
\displaystyle \frac{d\psi_{3}(t)}{dt}
=\psi_{1}(t)\big[\beta(1-u_{1}^{*}(t))x^{*}(t)\big]
+\psi_{3}(t)\mu +\chi_{[0,t_{f}-\tau]}(t)\psi_{2}(t+\tau)\big[\beta
(u^{*}_{1}(t+\tau)-1)x^{*}(t)\big],\\[0.3cm]
\displaystyle \frac{d\psi_{4}(t)}{dt}=1+\psi_{2}(t)py^{*}(t)
+\psi_{4}(t)\big[h-cx^{*}(t)y^{*}(t)\big],\\[0.3cm]
u^{*}_{1}=\min\bigg(1,\max\bigg(0,\frac{\beta}{A_{1}}\bigg[
\psi_{2}(t)v^{*}_{\tau}x^{*}_{\tau}
-\psi_{1}(t)v^{*}(t)x^{*}(t)\bigg]\bigg)\bigg),\\[0.3cm]
u^{*}_{2}=\min\bigg(1,\max\bigg(0,\frac{1}{A_{2}}
\psi_{3}(t)aNy^{*}(t)\bigg)\bigg),\\[0.3cm]
\psi_{i}(t_{f})=0, \quad i=1,\ldots,5.
\end{cases}
\end{equation*}

% --------------------------------------

\section{Numerical simulations} 
\label{sec:4}

In order to solve the optimality system given in Section~\ref{sec:3:3}, 
we use a numerical scheme based on forward and backward 
finite difference approximations. Precisely, 
we implemented Algorithm~\ref{our:alg}.

\begin{remark}
In our implementation, we choose the initial functions
\eqref{2}, satisfying \eqref{3}, as 
$x(t) \equiv x_0$, $y(t) \equiv y_0$, $v(t) \equiv v_0$, 
$z(t) \equiv z_0$, for all $t \in [-\tau, 0]$
(see Step~1 of Algorithm~\ref{our:alg}).
\end{remark}

% ---------------------------------
\begin{algorithm}
\caption{Numerical algorithm for the optimal control problem \eqref{sy3}}
\label{our:alg}

\medskip

\underline{Step 1}:

\medskip

for $i = -m, \ldots , 0,$ do:
$$
x_i = x_0, \quad y_i = y_0, 
\quad v_i = v_0, \quad z_i = z_0,
\quad u_1^i = 0,\quad u_2^i = 0
$$
end for;

for $i = n, \ldots , n + m$, do:
$$
\psi_1^i = 0, \quad \psi_2^i = 0, 
\quad \psi_3^i = 0, \quad \psi_4^i =0
$$
end for;

\medskip

\underline{Step 2}:

\medskip

for $i = 0$, \ldots , $n-1$, do:
\begin{equation*}
\begin{split}
x_{i+1} &= x_i + h[\lambda - d x_i - \beta(1-u_{1}^i)  x_i v_i],\\
y_{i+1} &= y_i + h[ \beta (1-u_{1}^i) x_{i-m}v_{i-m}-ay_i-py_i z_i],\\
v_{i+1} &= v_i + h[ aN(1-u_2^i)y_i-\mu v_i],\\
z_{i+1} &= z_i + h[ cx_i y_i z_i-hz_i],\\
\psi_1^{n-i-1} &=  \psi_1^{n-i} 
- h[ 1 + \psi_1^{n-i}(d +  (1-u_{1}^i)\beta v_{i+1}) 
- \psi_4^{n-i}(c y_{i+1} z_{i+1})\\
&\qquad +\chi_{[0,t_{f}-\tau]}(t_{n-i})
\psi_2^{n-i+m}(u_{1}^{i+m}-1)\beta v_{i+1}],\\
\psi_2^{n-i-1} 
&=\psi_2^{n-i} -h[\psi_2^{n-i}(a+p z_{i+1}) 
-\psi_2^{n-i}(aN(1-u_{2}^i))- \psi_4^{n-i}(c x_{i+1} z_{i+1}],\\
\psi_3^{n-i-1} 
&= \psi_3^{n-i} 
-h\big[\psi_1^{n-i}(1-u_1^i)\beta x_{i+1} + \psi_3^{n-i} \mu\\
&\qquad +\chi_{[0,t_{f}-\tau]}(t_{n-i})\psi_2^{n-i+m}(u_{1}^{i+m}-1) x_{i+1}\big],\\
\psi_4^{n-i-1} 
&= \psi_4^{n-i} - h[1+ p\psi_2^{n-i} y_{i+1} + \psi_4^{n-i}(h-cx_{i+1}y_{i+1})],\\
R_1^{i+1} &= (\beta/A_1)(\psi_2^{n-i-1}v_{i-m+1}x_{i-m+1}-\psi_1^{n-i-1}v_{i+1}x_{i+1}),\\
R_2^{i+1} &= (1/A_2)\psi_3^{n-i-1}aN y_{i+1},\\
u_1^{i+1} &=\min(1,\max(R_1^{i+1},0)),\\
u_2^{i+1} &=\min(1,\max(R_2^{i+1},0))
\end{split}
\end{equation*}
end for;

\medskip

\underline{Step 3}:

\medskip

for $i = 1, \ldots , n$, write:
$$
x^*(t_i) = x_i, \quad y^*(t_i) = y_i, \quad v^*(t_i) =v_i, 
\quad z^*(t_i) = z_i, \quad u_1^*(t_i) = u_1^i, 
\quad u_2^*(t_i) = u_2^i
$$
end for.
\end{algorithm}
% --------------------------------- 
In our simulations, the following parameters are used:
\begin{equation}
\label{eq:par:val}
\begin{gathered}
\lambda=1, \quad d=0.1, \quad \beta=0.00025, \quad  p=0.001,  \quad h=0.2,\\
a=0.2,  \quad c=0.03,  \quad \mu=3,  \quad A_1=30,  \quad A_2=40,  \quad  \tau=10.
\end{gathered}
\end{equation}
Such values respect the HIV parameter ranges given 
in Table~\ref{table:nonlin}.
% ---------------------------------
\begin{table}
\begin{center}
\caption{Parameters, their symbols and meaning, and default values used in HIV literature} 
\begin{tabular}{c c c c} \hline 
Parameters & \multicolumn{1}{c}{Meaning} &  Value &  References \\ \hline \hline 
$\lambda$ &  \multicolumn{1}{p{4cm}}{\raggedright source rate of CD4+ T cells} 
& $1$--$10$ cells $\mu l^{-1}$ days$^{-1}$ & \cite{crs}\\ 
$d$ & \multicolumn{1}{p{4cm}}{decay rate of healthy cells} 
& $0.007$--$0.1$ days$^{-1}$ & \cite{crs} \\
$\beta$   & \multicolumn{1}{p{4cm}}{rate at which CD4+ T cells become infected} 
& $0.00025$--$0.5$ $\mu l$ virion$^{-1}$ days$^{-1}$   & \cite{crs}\\
$a$ & \multicolumn{1}{p{4cm}}{death rate of infected CD4+ T cells, not by CTL} 
& $0.2$--$0.3$ days$^{-1}$ & \cite{crs} \\
$\mu$  & \multicolumn{1}{p{4cm}}{clearance rate of virus} 
& $2.06$--$3.81$ days$^{-1}$ & \cite{per}\\
$N$ & \multicolumn{1}{p{4cm}}{number of virions produced by infected CD4+ T-cells} 
& $6.25$--$23599.9$ virion$^{-1}$ & \cite{5a,32}\\
$p$ & \multicolumn{1}{p{4cm}}{clearance rate of infection} 
& $1$--$4.048 \times 10^{-4}$ ml virion days$^{-1}$ & \cite{5a,22} \\
$c$ & \multicolumn{1}{p{4cm}}{activation rate of CTL cells} 
& $0.0051$--$3.912$ days$^{-1}$ & \cite{5a} \\
$h$ & \multicolumn{1}{p{4cm}}{death rate of CTL cells} 
& $0.004$--$8.087$ days$^{-1}$ & \cite{5a} \\
$\tau$ & \multicolumn{1}{p{4cm}}{time delay} & 
$7$--$21$ days & \cite{14,17}\\ \hline 
\end{tabular}
\label{table:nonlin}
\end{center}
\end{table}
% ---------------------------------
In Figure~\ref{fig1}, we use the two following initial
conditions:
\begin{equation}
\label{ic1}
x_0=5, \quad  y_0=1, \quad v_0=1, \quad z_0=2
\end{equation}
and
\begin{equation}
\label{ic2}
x_0=45, \quad y_0=2, \quad v_0=1, \quad z_0=4.
\end{equation}
Moreover, besides parameters \eqref{eq:par:val}, we have chosen 
$N =750$, which means that $N\beta\lambda-d\mu = -0.1125 < 0$. 
According to Theorem~\ref{x}, the disease-free equilibrium 
$E_f = (10,0,0,0)$ is locally asymptotically stable. The plots of
Figure~\ref{fig1} confirm this result of local stability 
for both initial conditions \eqref{ic1} and \eqref{ic2}.
% ---------------------------------
\begin{figure}[!t]
\begin{center}
\subfloat[$x(t)$]{\label{fig1:a}
\includegraphics[scale=0.20]{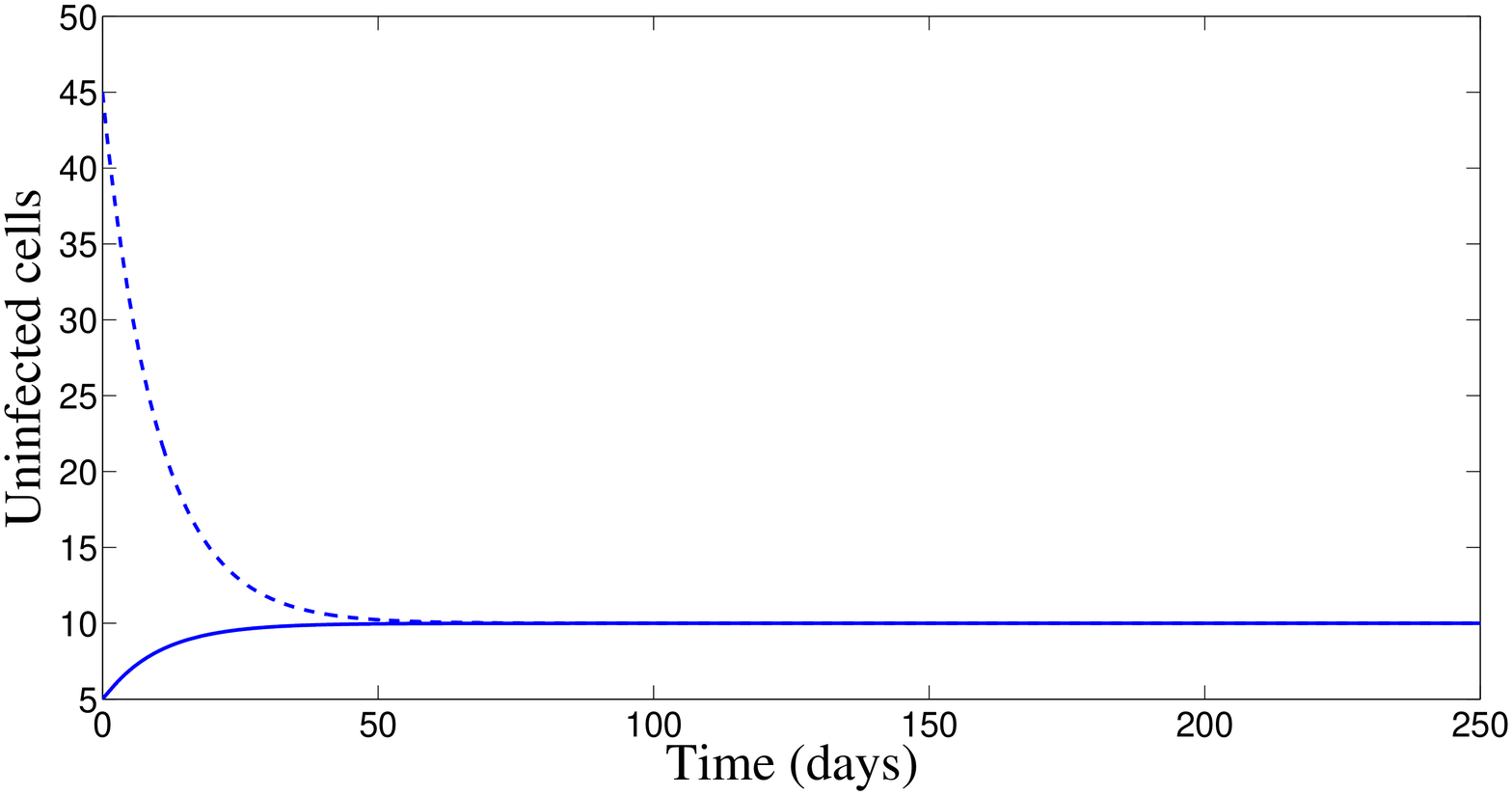}}
\subfloat[$y(t)$]{\label{fig1:b}
\includegraphics[scale=0.20]{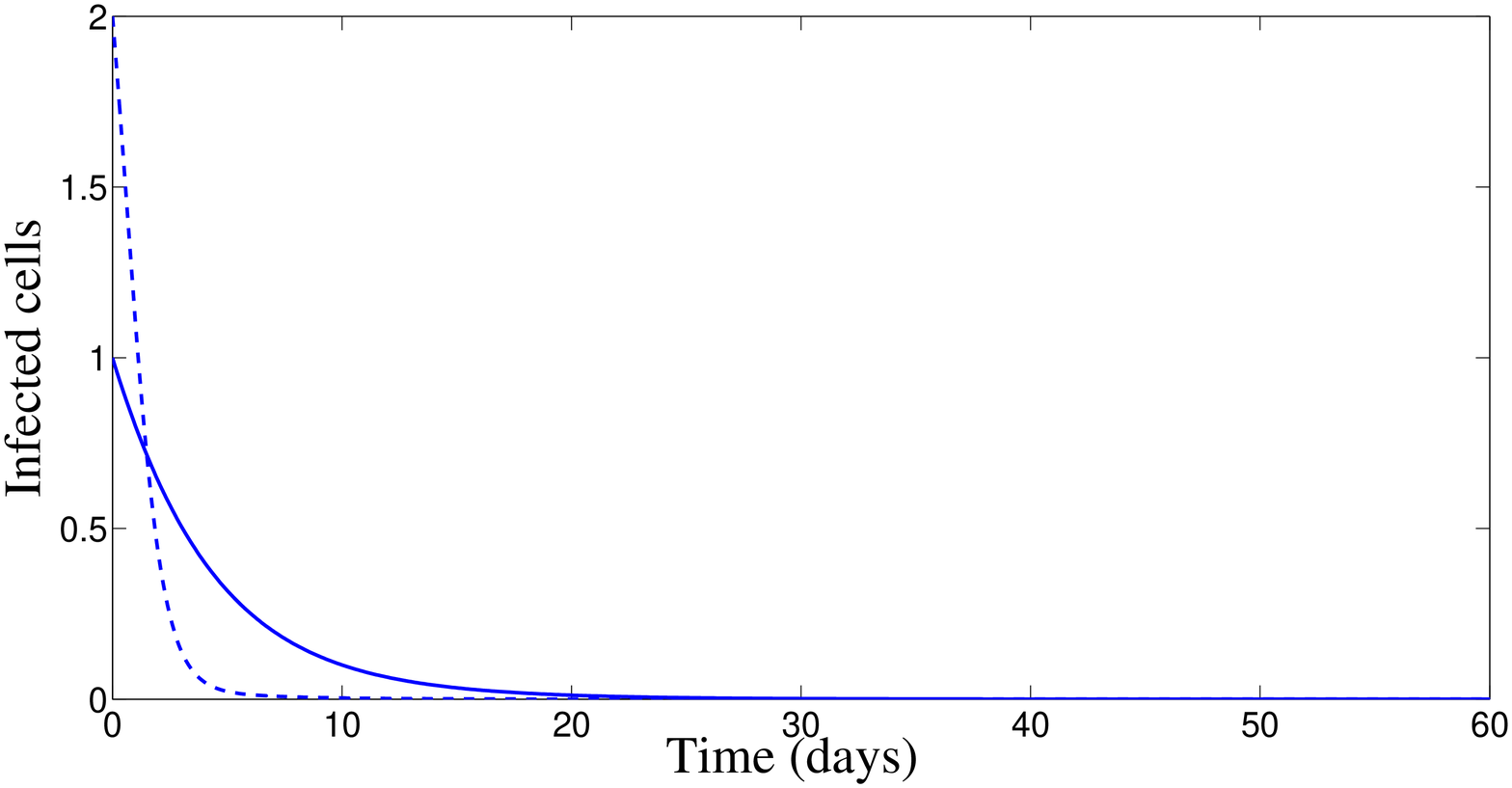}}\\ 
\subfloat[$v(t)$]{\label{fig1:c}
\includegraphics[scale=0.20]{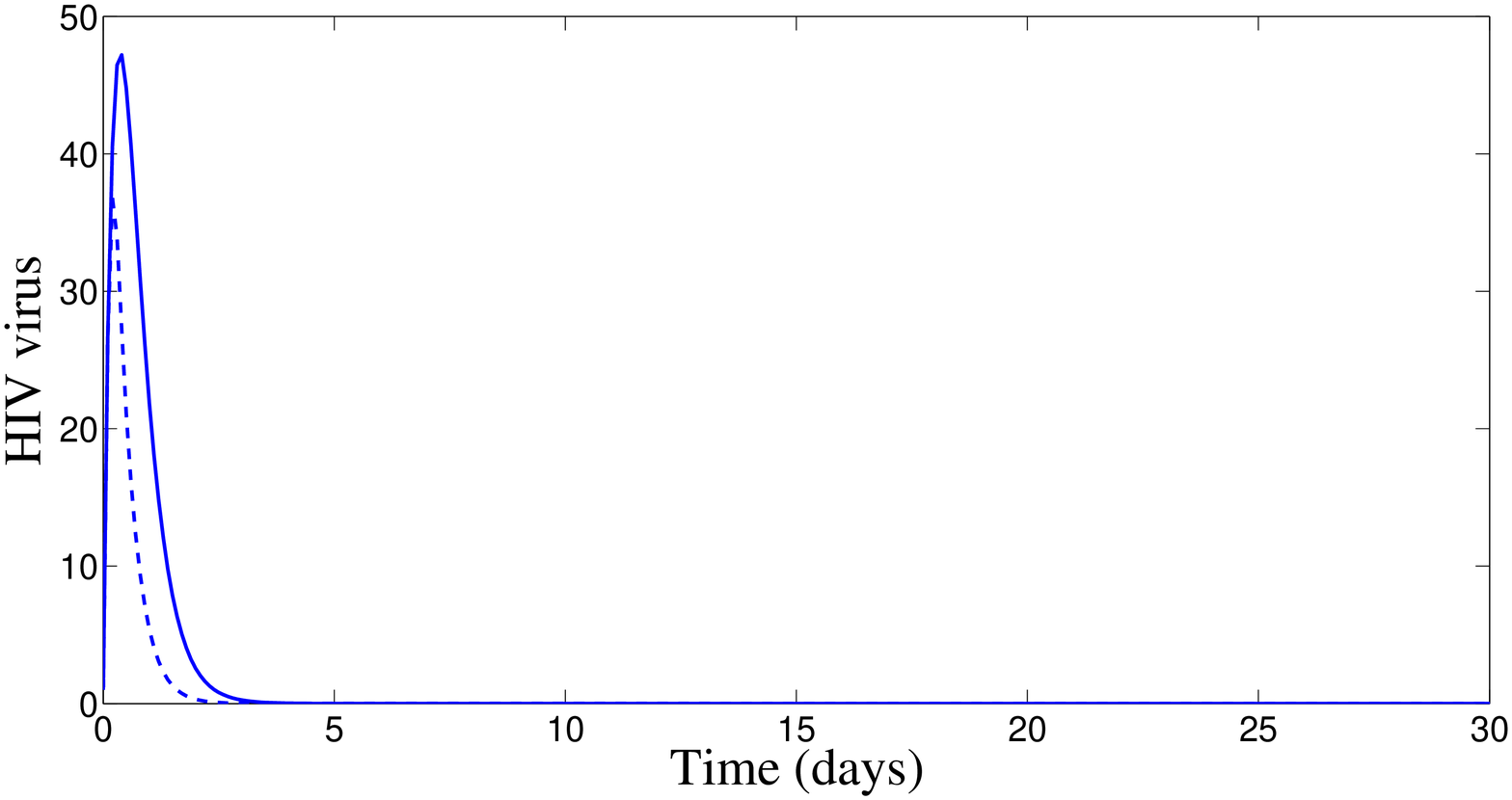}}
\subfloat[$z(t)$]{\label{fig1:d}
\includegraphics[scale=0.20]{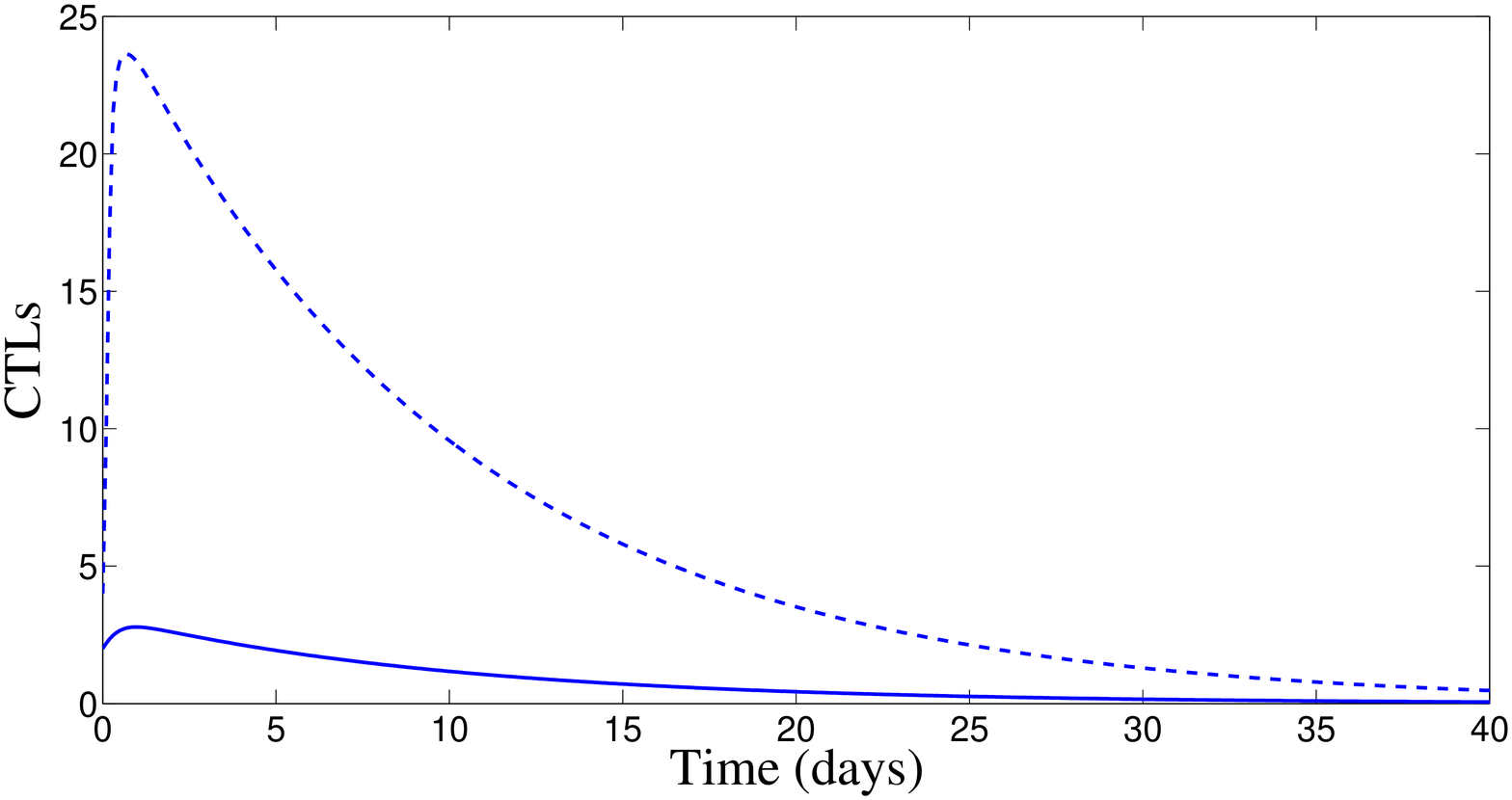}}
\caption{Behavior of infection for
$N=750$ and parameter values \eqref{eq:par:val}.
The continuous curves correspond to initial conditions \eqref{ic1}
and dashed curves to initial conditions \eqref{ic2}.}
\label{fig1}
\end{center} 
\end{figure}
% ---------------------------------
In Figure~\ref{fig2}, we have chosen $N = 1500$, 
which means that $\lambda \mu c-\beta aN h = 7.50 \times 10^{-2} > 0$ 
and $N \beta(\lambda \mu c- \beta h aN)-(\mu^2  c d) = 11.245 \times 10^{-4}> 0$.
According to Theorem~\ref{y}, the endemic equilibrium $E_2$
is locally stable. Figure~\ref{fig2} confirms this result
numerically: we clearly see the convergence to the equilibrium point 
$E_2 = (8.33,0.8,80,8.333)$. However, it is interesting to point out 
that, with control, a significant decrease of the infected cells, 
free viruses, and CTL cells, is observed (see Figure~\ref{fig2}). 
The uninfected cells get maximized. 
It is worth to mention that with control treatment the infection dies 
very fast and the dynamics goes toward the disease-free equilibrium. 
% ---------------------------------
\begin{figure}[!t]
\begin{center}
\subfloat[$x(t)$]{\label{fig2:a}
\includegraphics[scale=0.20]{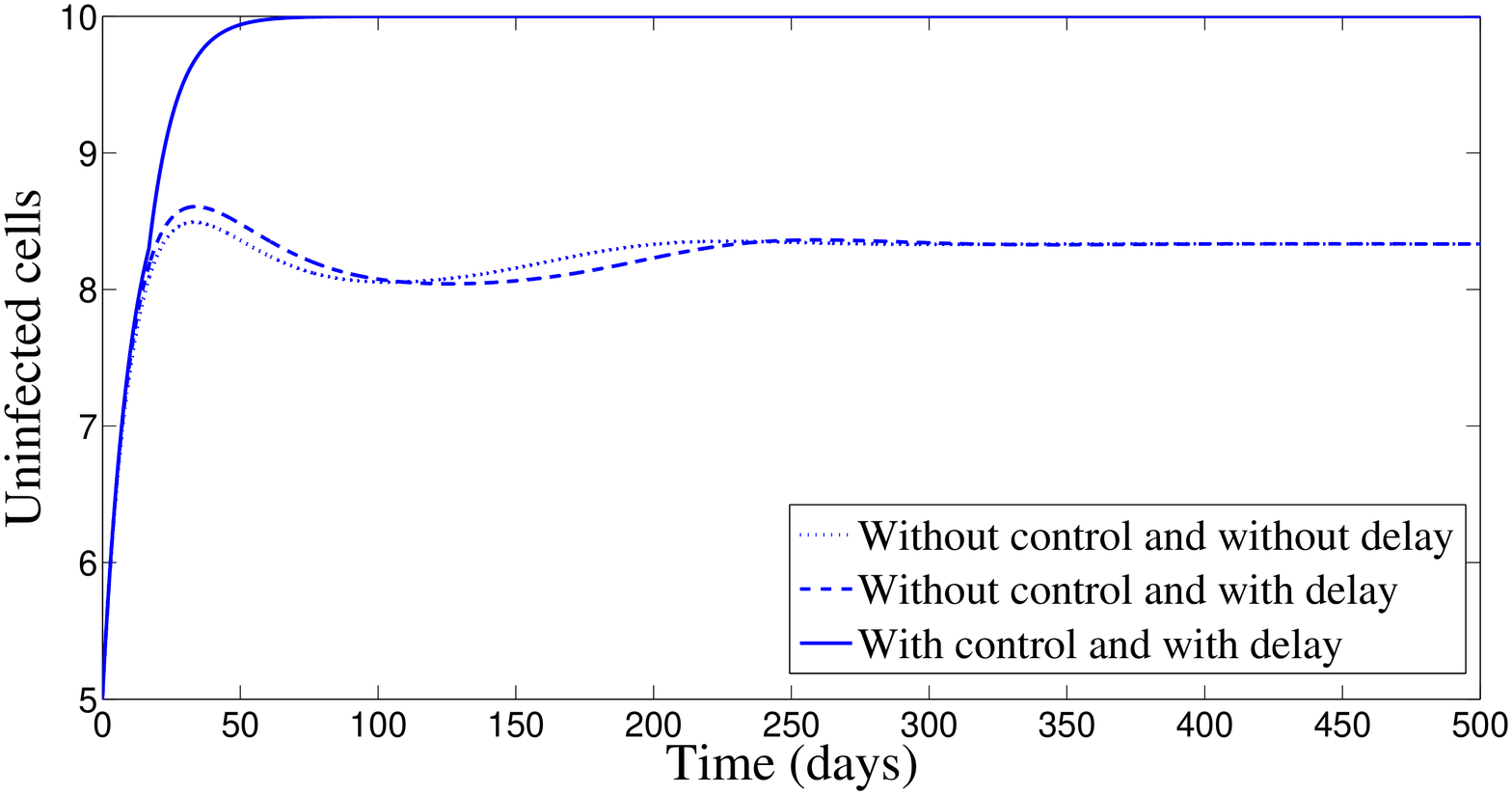}}
\subfloat[$y(t)$]{\label{fig2:b}
\includegraphics[scale=0.20]{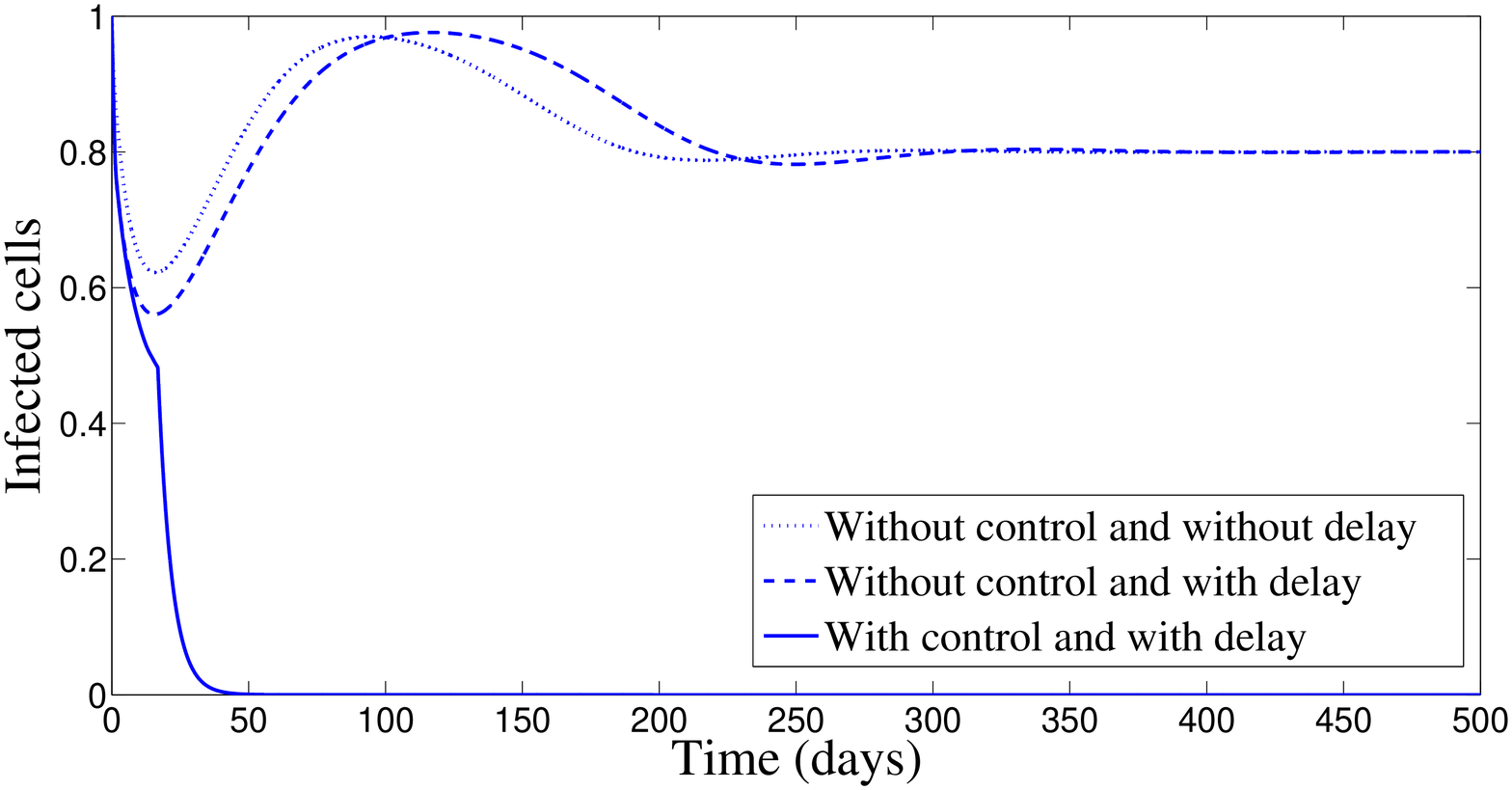}}\\
\subfloat[$v(t)$]{\label{fig2:c}
\includegraphics[scale=0.20]{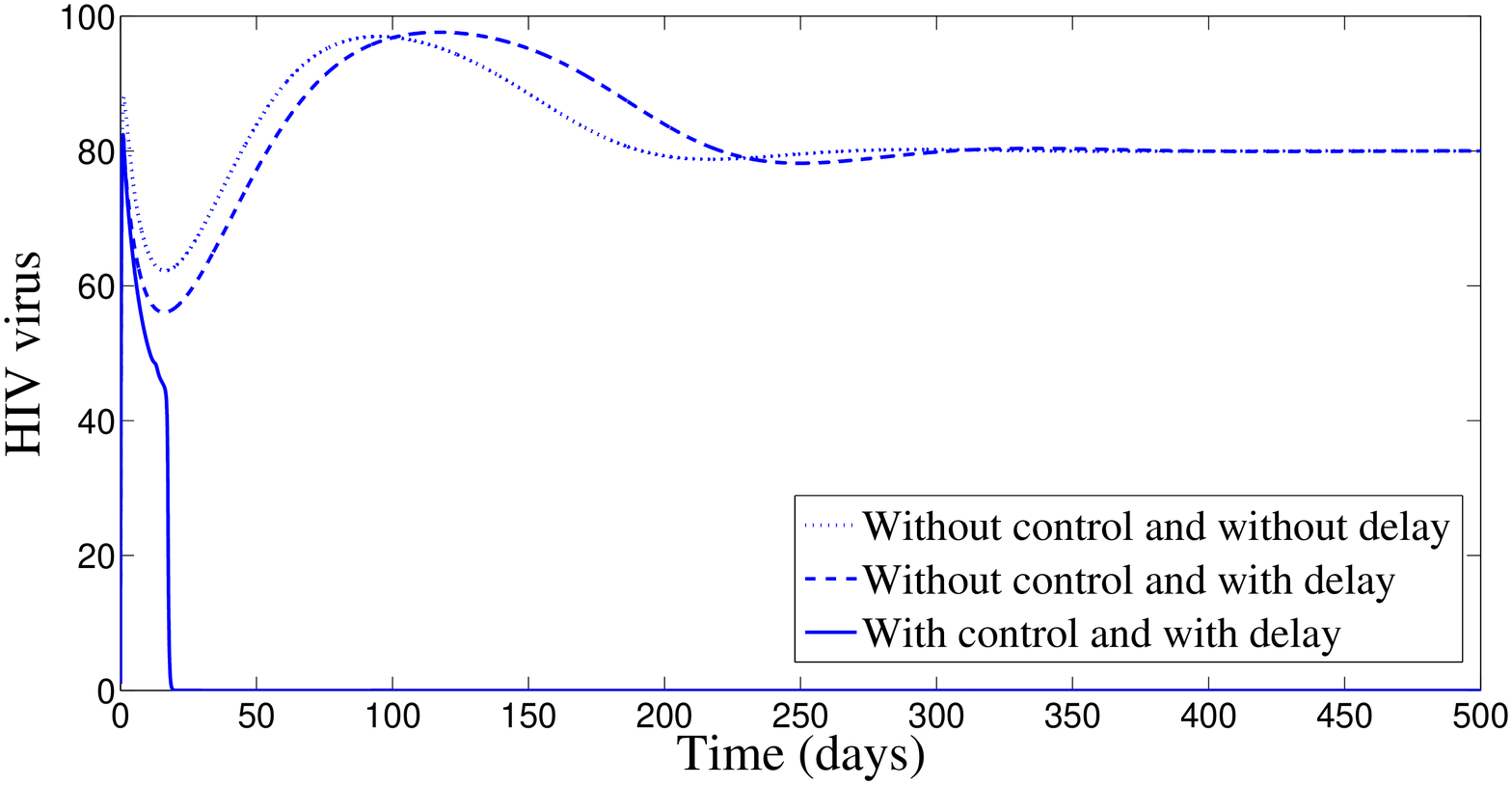}}
\subfloat[$z(t)$]{\label{fig2:d}
\includegraphics[scale=0.20]{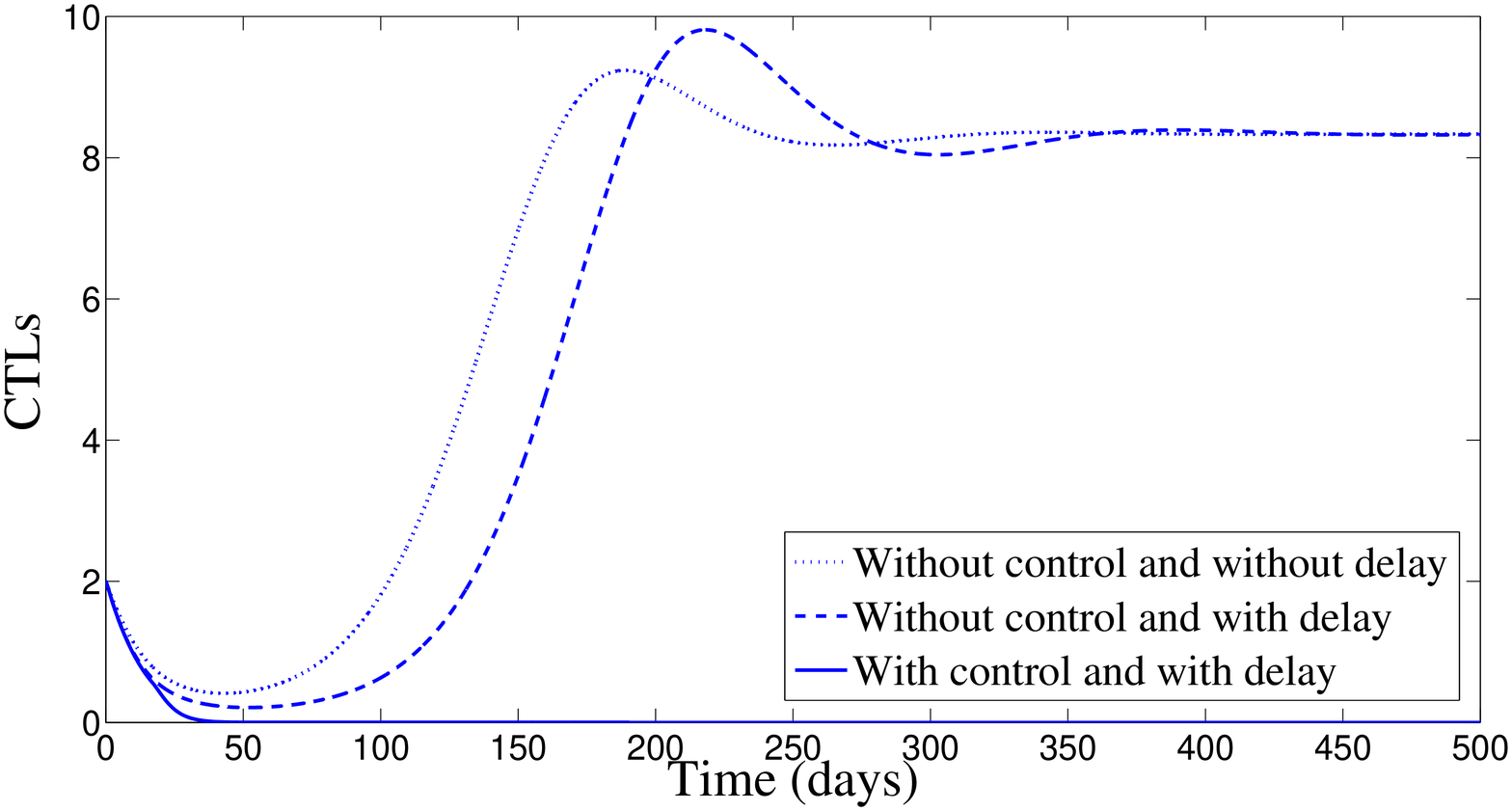}}
\caption{Behavior of infection during 500 days 
for $N=1500$, parameter values \eqref{eq:par:val},
and initial conditions \eqref{ic1}: dotted line 
without control and without delay ($\tau = 0$); 
dashed line without control but with delay $\tau = 10$; 
continuous line for the delayed problem $\tau = 10$ 
under optimal control.}
\label{fig2}
\end{center}
\end{figure}
% ---------------------------------
\begin{figure}[!t]
\begin{center}
\subfloat[$u_{1}^{*}(t)$]{\label{fig3:a}
\includegraphics[scale=0.20]{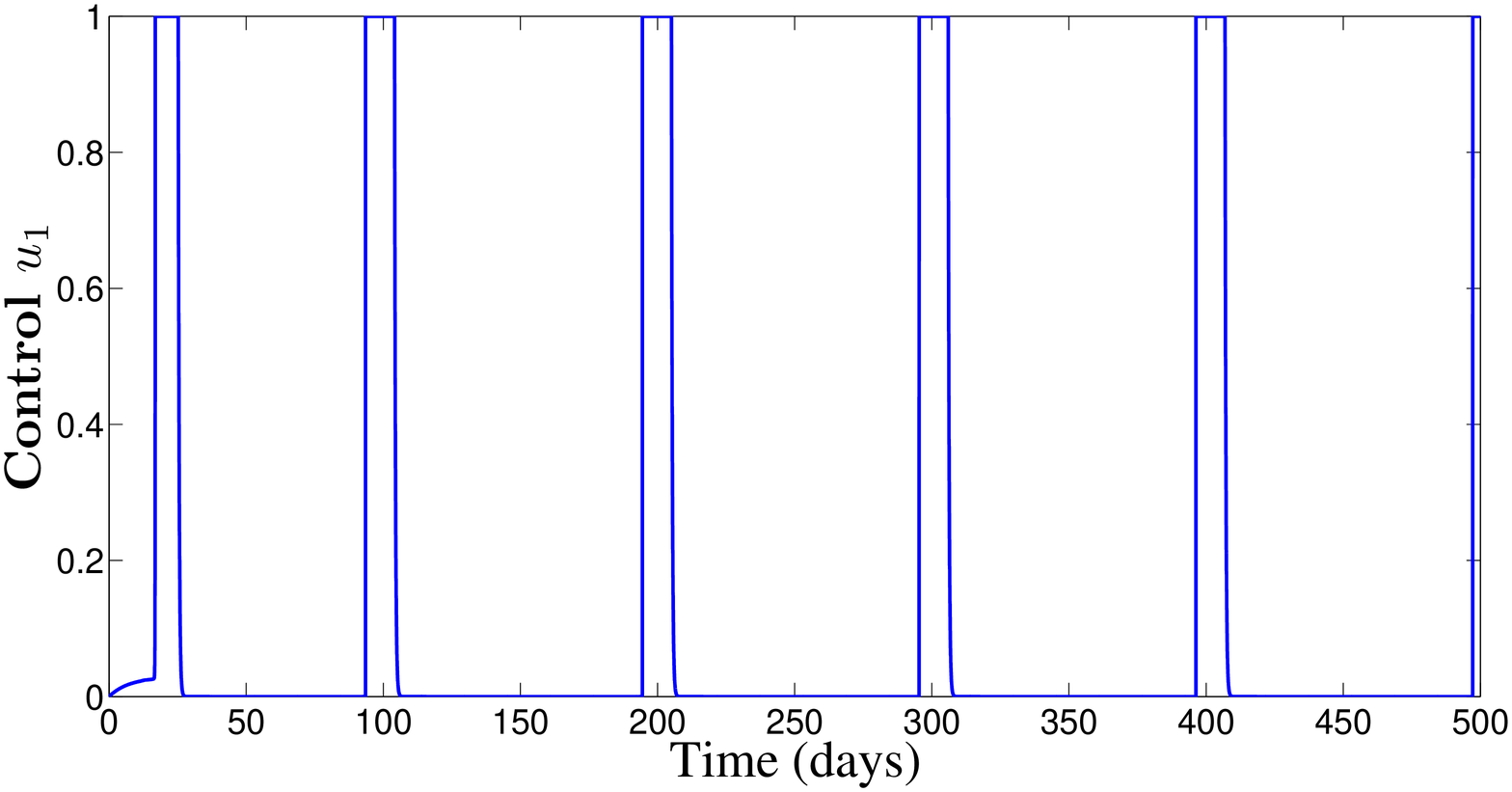}}
\subfloat[$u_{2}^{*}(t)$]{\label{fig3:b}
\includegraphics[scale=0.20]{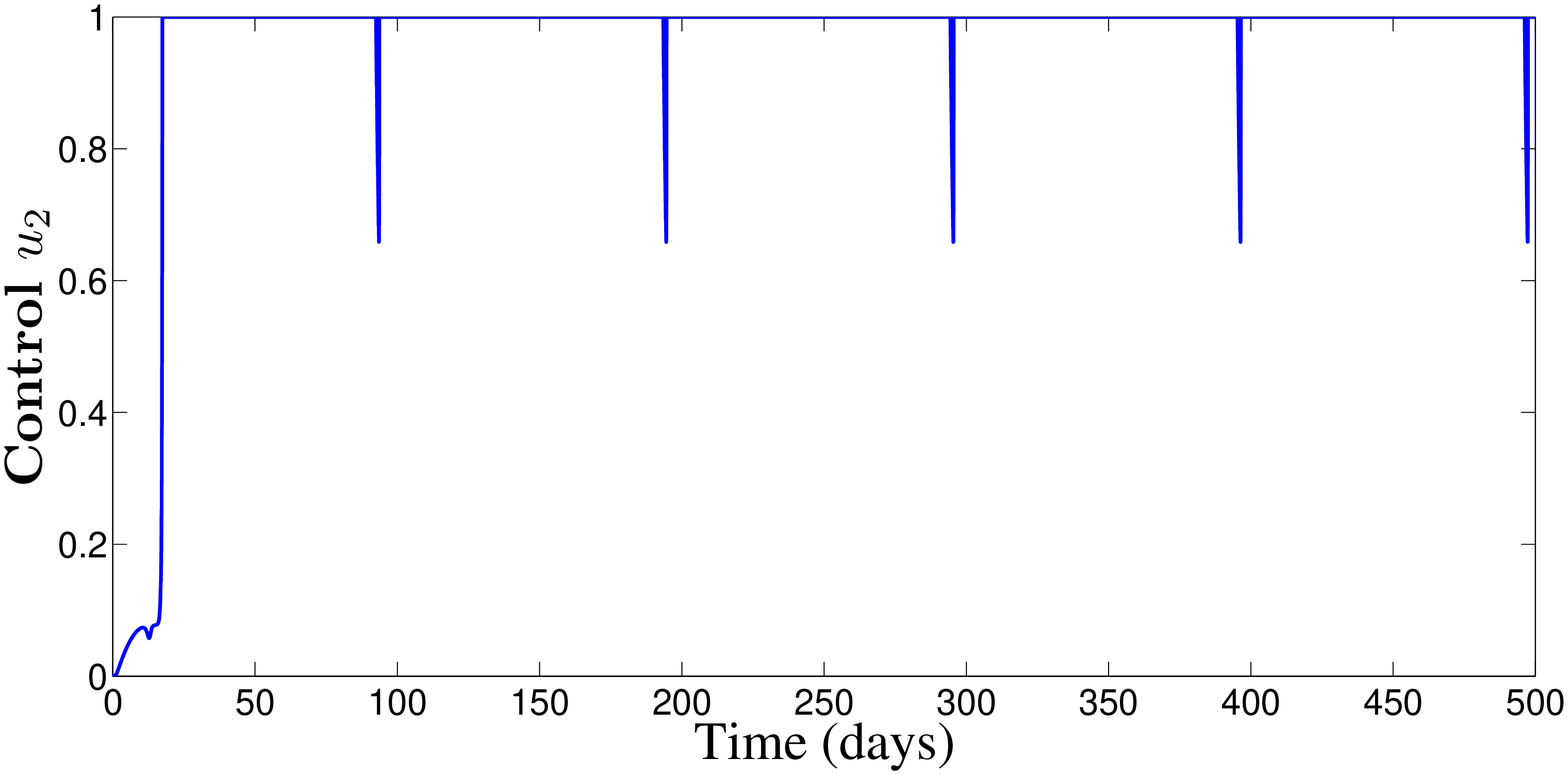}}
\caption{The two extremal controls $u_{1}^{*}$ and $u_{2}^{*}$ 
for the delayed optimal control problem \eqref{sy3} with $N=1500$, 
parameter values \eqref{eq:par:val}, and initial conditions \eqref{ic1},
obtained using Algorithm~\ref{our:alg}.}
\label{fig3}
\end{center}
\end{figure}
% ---------------------------------
The behavior of the two treatments during
time is given in Figure~\ref{fig3}. We can see 
that the first control makes several
switchings from zero to one and vice versa. We can understand
from this that one can manage the first treatment to the patient
periodically in full manner. In this figure, we can also see
that the second control is almost equal to one but with sudden
decreases, for some short periods of time, without vanishing.

% --------------------------------------

\section{Conclusion}
\label{sec:5}

In this work we have studied a delayed HIV viral infection model
with CTL immune response. The considered model includes four
differential equations describing the interaction between the
uninfected cells, infected cells, HIV free viruses and CTL immune
response. An intracellular time delay and two treatments are
incorporated to the suggested model. First, existence,
positivity and boundedness of solutions are established. Next, an
optimization problem is formulated in order to search the better
optimal control pair to maximize the number of uninfected
cells, reduce the infected cells and minimize the viral load.
Two control functions are incorporated in the model, which
represent the efficiency of drug treatment in inhibiting viral
production and preventing new infections. The existence of such
optimal control pair is established and the optimality system is
given and solved numerically, using a forward and backward
difference approximation scheme. It was shown that the obtained
optimal control pair increases considerably the number of CD4+
cells while reducing the number of infected. Moreover, it
was also observed that, under optimal control, the viral load decreases
significantly compared with the model without control, which will
improve the life quality of the patient. It remains an open question
how to prove stability of the endemic equilibrium $E_2$ of model 
\eqref{sys} for an arbitrary intracellular time delay $\tau > 0$.

% --------------------------------------

\begin{acknowledgements}
This research is part of second author's Ph.D., which is carried
out at University of Hassan II Casablanca, Morocco. 
Torres was partially supported by project TOCCATA, 
reference PTDC/EEI-AUT/2933/2014, funded by Project 
3599 -- Promover a Produ\c{c}\~ao Cient\'{\i}fica e Desenvolvimento
Tecnol\'ogico e a Constitui\c{c}\~ao de Redes Tem\'aticas (3599-PPCDT)
and FEDER funds through COMPETE 2020, Programa Operacional
Competitividade e Internacionaliza\c{c}\~ao (POCI), and by Portuguese
funds through Funda\c{c}\~ao para a Ci\^encia e a Tecnologia (FCT)
and CIDMA, within project UID/MAT/04106/2013. The authors are grateful
to the referees for their valuable comments and helpful suggestions.
\end{acknowledgements}

% --------------------------------------

% --------------------------------------

\end{document}